\documentclass[11pt,oneside]{amsart}

\usepackage[utf8]{inputenc}  
\usepackage[T1]{fontenc} 

\usepackage[english]{babel}

\usepackage{hyperref}

\allowdisplaybreaks

\usepackage{vmargin}
\setmarginsrb{3cm}{2cm}{3cm}{2cm}{0cm}{2cm}{0cm}{1cm}

\usepackage{tikz}
\usetikzlibrary{automata,arrows.meta}
\definecolor{MyGreen}{rgb}{0.13,0.55,0.13}

\theoremstyle{plain}
\newtheorem{theorem}{Theorem}
\newtheorem{lemma}[theorem]{Lemma}
\newtheorem{corollary}[theorem]{Corollary}
\newtheorem{proposition}[theorem]{Proposition}
\theoremstyle{definition}

\newtheorem{remark}[theorem]{Remark} 
\numberwithin{equation}{section}

\newcommand{\R}{\mathbb R}
\newcommand{\N}{\mathbb N}

\newcommand{\Q}{\mathbb Q}
\newcommand{\pr}{\beta}

\newcommand{\B}{\boldsymbol{\beta}}
\newcommand{\Dig}{\mathrm{Dig}}
\newcommand{\Per}{\mathrm{Per}}
\newcommand{\ceil}[1]{\left\lceil#1\right\rceil}
\newcommand{\floor}[1]{\lfloor#1\rfloor}
\newcommand{\Int}{[\![0,p-1]\!]}
\newcommand{\DB}{d_{\boldsymbol{\beta}}}
\newcommand{\DBi}[1]{d_{{\B}^{(#1)}}}

\newcommand{\qDBi}[1]{d_{\boldsymbol{\beta}^{(#1)}}^{*}}

\title{On periodic alternate base expansions}
\author{Émilie Charlier$^*$, Célia Cisternino and Savinien Kreczman\\
\\
\tiny{Department of Mathematics,
University of Li\`ege,\\
All\'ee de la D\'ecouverte 12,
4000 Li\`ege, Belgium
}}
\thanks{\emph{E-mail address:} \texttt{echarlier@uliege.be}, \texttt{celiacisternino@gmail.com} and \texttt{savinien.kreczman@uliege.be}\\
$^*$Corresponding author.}

\begin{document}

\begin{abstract}
For an alternate base $\B=(\beta_0,\ldots,\beta_{p-1})$, we show that if all rational numbers in the unit interval $[0,1)$ have periodic expansions with respect to the $p$ shifts of $\B$, then the bases $\beta_0,\ldots,\beta_{p-1}$ all belong to the extension field $\Q(\pr)$ where $\beta$ is the product $\beta_0\cdots\beta_{p-1}$ and moreover, this product $\beta$ must be either a Pisot number or a Salem number. We also prove the stronger statement that if the bases $\beta_0,\ldots,\beta_{p-1}$ belong to $\Q(\pr)$ but the product $\beta$ is neither a Pisot number nor a Salem number then the set of rationals having an ultimately periodic $\B$-expansion is nowhere dense in $[0,1)$. Moreover, in the case where the product $\beta$ is a Pisot number and the bases $\beta_0,\ldots,\beta_{p-1}$ all belong to $\Q(\beta)$, we prove that the set of points in $[0,1)$ having an ultimately periodic $\B$-expansion is precisely the set $\Q(\beta)\cap[0,1)$. For the restricted case of Rényi real bases, i.e., for $p=1$ in our setting, our method gives rise to an elementary proof of Schmidt's original result. Therefore, even though our results generalize those of Schmidt, our proofs should not be seen as generalizations of Schmidt's original arguments but as an original method in the generalized framework of alternate bases, which moreover gives a new elementary proof of Schmidt's results from 1980. 
As an application of our results, we show that if $\B=(\beta_0,\ldots,\beta_{p-1})$ is an alternate base such that the product $\beta$ of the bases is a Pisot number and $\beta_0,\ldots,\beta_{p-1}\in\Q(\beta)$, then $\B$ is a Parry alternate base, meaning that the quasi-greedy expansions of $1$ with respect to the $p$ shifts of the base $\B$ are ultimately periodic. 
\end{abstract}

\maketitle

\bigskip
\hrule
\bigskip



\bigskip
\hrule
\bigskip

\section{Introduction}

The real base expansions of real numbers were introduced by Rényi~\cite{Renyi1957}. Given a real base $\beta>1$, a representation of a real number $x\in[0,1)$ is an infinite sequence $(a_n)_{n\in\N}$ of non-negative integer digits such that $x=\sum_{n= 0}^\infty\frac{a_n}{\beta^{n+1}}$. Choosing at each step the largest possible digit $a_n$ so that the partial sum $\sum_{k= 0}^n\frac{a_k}{\beta^{k+1}}$ does not exceed $x$, we obtain one particular $\beta$-representation of $x$ called the \emph{$\beta$-expansion of $x$} and denoted by $d_\beta(x)$. Rényi observed that the digits of the $\beta$-expansion of $x$ can also obtained by iterating the so-called \emph{$\beta$-transformation} $T_\beta\colon[0,1)\to [0,1),\ x\mapsto \beta x-\floor{\beta x}$, where $\floor{\cdot}$ denotes the floor function. More precisely, the computation of the $n$-th digit is given by the formula $a_n=\floor{\beta T_\beta^n(x)}$. Then Rényi showed, among other things, that the map $T_\beta$ defines an ergodic dynamical system. The dynamical properties of the $\beta$-expansions were extensively studied since the seminal work of Rényi. 

In particular, the \emph{$\beta$-shift} $S_\beta$ received a lot of attention. This set is defined as the topological closure (with respect to the product topology on infinite words) of the set $\{d_\beta(x) :  x\in[0,1)\}$. It is shift invariant and it defines a dynamical system that is measure theoretically isomorphic to the dynamical system built on $T_\beta$. Parry provided a combinatorial characterization of elements in the $\beta$-shift~\cite{Parry1960} involving one particular infinite word $d_\beta^*(1)$, which is nowadays called the \emph{quasi-greedy $\beta$-expansion of $1$} and which is defined as the limit of the sequences $d_\beta(x)$ as $x$ tends to $1^-$, that is, $d_\beta^*(1)=\lim_{x\to 1^-}d_\beta(x)$. Ito and Takahashi then showed that the $\beta$-shift $S_\beta$ is of finite type (a property they call \emph{markovian}) if and only if $d_\beta^*(1)$ is purely periodic~\cite{Ito&Takahashi:1974}. Further, Bertrand-Mathis showed that the $\beta$-shift $S_\beta$ is sofic if and only if $d_\beta^*(1)$ is ultimately periodic~\cite{Bertrand-Mathis1986}. From these results, we see the importance of the particular infinite word $d_\beta^*(1)$ in the study of $\beta$-expansions of Rényi. Nowadays, real bases $\beta$ such that $d_\beta^*(1)$ is ultimately periodic are called \emph{Parry numbers}. 

In~\cite{Schmidt1980}, Schmidt studied the set $\Per(\beta)$ of ultimately periodic points of the $\beta$-transfor\-mation $T_\beta$. In particular, his results imply that all Pisot numbers, i.e., algebraic integers $\beta>1$ whose Galois conjugates (that is, the roots of the minimal polynomial of $\beta$) distinct from $\beta$ all have modulus less than $1$, are Parry numbers. The aim of the present paper is to understand the set of real numbers $x\in[0,1)$ having an ultimately periodic alternate base expansion.

Alternate base expansions of real numbers are a generalization of Rényi $\beta$-expansions~\cite{CharlierCisternino2021}. We give here the necessary background in order to state the generalization of Schmidt's result that we seek. An \emph{alternate base} $\B=(\beta_0,\ldots,\beta_{p-1})$ is a $p$-tuple of real bases, that is, $\beta_i>1$ for every $i\in\Int$ (throughout this text, an interval of integers $\{i, \ldots , j\}$ with $i\le j$ is denoted $[\![i,j]\!]$). A \emph{$\B$-representation} of a real number $x$ is an infinite sequence $a=(a_n)_{n\in \N}$ of integers such that 
\begin{equation}
\label{Eq : valueAlternatBase}
	x=\sum_{m=0}^\infty \sum_{i=0}^{p-1} \frac{a_{mp+i}}{(\beta_0\cdots\beta_{p-1})^m\beta_0\cdots\beta_i}.
\end{equation}
We use the convention that for all $n\in \N$, $\beta_n=\beta_{n \bmod p}$ and 
\[
	\B^{(n)}=(\beta_n,\ldots,\beta_{n+p-1}).
\]

For $x\in[0,1)$, a distinguished $\B$-representation $(\varepsilon_n)_{n\in \N}$, called the \emph{$\B$-expansion} of $x$, is obtained from the \emph{greedy algorithm}: set $r_0=x$ and, for $n\in \N$, $\varepsilon_n=\floor{\beta_n r_n}$ and $r_{n+1}=\beta_n r_n-\varepsilon_n$. The $\B$-expansion of $x$ is denoted $\DB(x)$. The $n$-th digit $\varepsilon_n$ belongs to $[\![0,\ceil{\beta_n}-1]\!]$. The number $r_n$ is called the $n$-th \emph{remainder} computed by the greedy algorithm. Note that the remainders all belong to $[0,1)$.

We let $\Per(\B)$ denote the set of real numbers in $[0,1)$ having an ultimately periodic greedy $\B$-expansion, that is, 
\begin{equation}\label{Eq : DefPerB}
\Per(\B)=\{ x\in [0,1) : \DB(x) \text{ is ultimately periodic}\}.
\end{equation}
As in the real base case, the digits of the $\B$-expansion may also be obtained by iterating a well-chosen transformation $T_{\B}$ \cite{CharlierCisterninoDajani2023}. The set $\Per(\B)$ may then be seen, up to some technicalities, as the set of ultimately periodic points of this map $T_{\B}$.

The main goal of this paper is to prove the following result generalizing Schmidt's theorems~\cite[Theorems 2.4 and 3.1]{Schmidt1980}. Recall that a \emph{Salem number} is an algebraic integer $\beta>1$ whose Galois conjugates distinct from $\beta$ all have modulus less than or equal to $1$, with equality for at least one of them.

\begin{theorem}
\label{Thm : OurSchmidt}
Let $\B=(\beta_0,\ldots,\beta_{p-1})$ be an alternate base and set $\pr=\prod_{i=0}^{p-1}\beta_i$. 
\begin{enumerate}
\item If $\Q\cap [0,1)\subseteq \bigcap_{i=0}^{p-1}\Per(\B^{(i)})$ then $\beta_0,\ldots,\beta_{p-1} \in \Q(\pr)$ and $\pr$ is either a Pisot number or a Salem number.
\label{OurSchmidt:1}
\item If $\pr$ is a Pisot number and $\beta_0,\ldots,\beta_{p-1} \in \Q(\pr)$ then $ \Per(\B)=\Q(\pr)\cap [0,1)$.
\label{OurSchmidt:2}
\end{enumerate}
\end{theorem}

Our proof of Theorem~\ref{Thm : OurSchmidt} is based on algebraic tools such as the alternate base spectrum defined in~\cite{CharlierCisterninoMasakovaPelantova2023} as a generalization of the $\beta$-spectrum originally introduced by Erd\H os, Jo\'o and Komornik~\cite{ErdosJooKomornik1990}. A crucial step in our method is a generalization of a result from~\cite{CharlierCisterninoMasakovaPelantova2023}, which is Theorem~\ref{Thm : GeneralizationPrague} below. In the reduced case of one real base, we obtain a proof that is much shorter than Schmidt's original one from~\cite{Schmidt1980}. Therefore, even though the statement of Theorem~\ref{Thm : OurSchmidt} generalizes Schmidt's results, the proof we give should not be seen as a generalization of Schmidt's original arguments. The new elementary proof of Schmidt's results that we obtain will be provided explicitly in Section~\ref{Sec:p=1}.
 
Note that the algebraic condition that the bases $\beta_0,\ldots,\beta_{p-1}$ all belong to the extension field $\Q(\pr)$ is trivially satisfied whenever $p=1$, that is, in the original case of Rényi's expansions. This condition already appeared in~\cite{CharlierCisterninoMasakovaPelantova2023}. In that work, the aim was to obtain algebraic descriptions of alternate bases $\B=(\beta_0,\dots,\beta_{p-1})$ defining a sofic \emph{$\B$-shift}, where for an alternate base $\B$, the $\B$-shift is defined as the topological closure of the set $\bigcup_{i=0}^{p-1}\{\DBi{i}(x) \colon x\in [0,1)\}$. In~\cite{CharlierCisternino2021}, a combinatorial characterization of the sofic $\B$-shifts was obtained; namely, the $\B$-shift is sofic if and only if $\qDBi{i}(1)$ is eventually periodic for every $i\in\Int$, where $\qDBi{i}(1)=\lim_{x\to 1^-}\DBi{i}(x)$. It was then shown in~\cite{CharlierCisterninoMasakovaPelantova2023} that a necessary condition for the $\B$-shift to be sofic is that the product $\pr=\prod_{i=0}^{p-1}\beta_i$ is an algebraic integer and all of the bases $\beta_0,\ldots,\beta_{p-1}$ belong to $\Q(\pr)$. On the other hand, it was also shown that if the product $\pr$ is a Pisot number and $\beta_0,\ldots,\beta_{p-1}\in\Q(\pr)$ then the $\B$-shift is sofic.

In view of the previously mentioned results describing the $\B$-shift, we call $\B$ a \emph{Parry alternate base} if $\qDBi{i}(1)$ is eventually periodic for every $i\in\Int$. As a direct consequence of Theorem~\ref{Thm : OurSchmidt}, we reobtain the above-mentioned result from~\cite{CharlierCisterninoMasakovaPelantova2023} generalizing the fact that all Pisot numbers are Parry numbers. Note that in the case $p=1$, this result was independently obtained by both Bertrand and Schmidt~\cite{Bertrand1977,Schmidt1980}.

\begin{corollary}
\label{Cor : PisotimpliqueParry}
Let $\B=(\beta_0,\ldots,\beta_{p-1})$ be an alternate base and set $\pr=\prod_{i=0}^{p-1}\beta_i$. If $\pr$ is a Pisot number and $\beta_0,\ldots,\beta_{p-1} \in \Q(\pr)$ then $\B$ is a \emph{Parry alternate base}.
\end{corollary}

We also prove the following theorem generalizing~\cite[Theorem 2.5]{Schmidt1980}. This result is a refinement of the item~\eqref{OurSchmidt:1} of Theorem~\ref{Thm : OurSchmidt}.

\begin{theorem}
\label{Thm:stronger}
Let $\B=(\beta_0,\ldots,\beta_{p-1})$ be an alternate base such that $\beta_0,\ldots,\beta_{p-1}\in\Q(\pr)$ and set $\pr=\prod_{i=0}^{p-1} \beta_i$. If $\beta$ is an algebraic integer that is neither a Pisot number nor a Salem number then $\Per(\B)\cap\Q$ is nowhere dense in $[0,1)$.
\end{theorem}

The paper has the following organization.
Section~\ref{Sec:prepa} contains the preparatory work in order to prove Theorems~\ref{Thm : OurSchmidt} and~\ref{Thm:stronger}. We start by proving two equivalent conditions on $x\in[0,1)$ in order to belong to $\Per(\B)$, allowing us to use the remainders of the greedy algorithms only once out of $p$ steps. Then we obtain a generalization of a result from~\cite{CharlierCisterninoMasakovaPelantova2023}, which is Theorem~\ref{Thm : GeneralizationPrague}. This generalization represents a crucial argument in the proof of our main result. In Section~\ref{Sec:OurSchmidt}, we prove Theorem~\ref{Thm : OurSchmidt} and Corollary~\ref{Cor : PisotimpliqueParry}.
Section~\ref{Sec:stronger} is dedicated to the proof of Theorem~\ref{Thm:stronger}. We then consider the restricted case where $p=1$, i.e., of usual real bases, in Section~\ref{Sec:p=1}. 
Several arguments used for the general case of alternate bases just vanish in this restricted setting, giving rise to a new proof of Schmidt's result that is simpler than the original one  from~\cite{Schmidt1980}. Finally, in Section~\ref{Sec:future}, we present some perspectives for future research.

\section{Preparatory results}
\label{Sec:prepa}

\subsection{On $\B$-periodicity and $\pr$-periodicity $p$ by $p$ steps}

From now on, we fix a positive integer $p$ and an alternate base $\B=(\beta_0,\ldots,\beta_{p-1})$, and we set $\pr=\prod_{i=0}^{p-1}\beta_i$. We define
\[
	f_{\B}\colon 
	\prod_{i=0}^{p-1}\ [\![0,\ceil{\beta_i}-1]\!]\to\R,\ 
	(a_0,\ldots, a_{p-1}) 
	\mapsto \sum_{i=0}^{p-1}a_i \beta_{i+1}\cdots\beta_{p-1}.
\]
and we consider the image of the map $f_{\B}$ as a set of real digits which we denote by $\Dig(\B)$, that is,
\begin{equation}
\label{Eq:Dig}
	\Dig(\B)
	=\left\{\sum_{i=0}^{p-1}a_i \beta_{i+1}\cdots\beta_{p-1} :
	\forall i\in \Int,\ a_i\in [\![0,\ceil{\beta_i}-1]\!]\right\}.
\end{equation}
These objects were already considered in~\cite{CharlierCisterninoDajani2023}. For $x\in[0,1)$ and $m\in \N$, we set 
\begin{equation}
\label{Eq : eta}
	\eta_m=f_{\B}(\varepsilon_{mp},\varepsilon_{mp+1},\ldots,\varepsilon_{mp+p-1})
\end{equation}
where $(\varepsilon_n)_{n\in\N}=\DB(x)$. Thus, we have $\eta_m\in\Dig(\B)$ and
\begin{equation}
\label{Eq : x sum eta}
	x=\sum_{m=0}^{\infty}\frac{\eta_m}{\pr^{m+1}}.
\end{equation}
We will sometimes write $\eta_m(x)$ instead of just $\eta_m$ when the dependence in $x$ needs to be emphasized. Moreover, for all $m\in\N$, we have 
\begin{equation}
\label{Eq : rn+1 rn}
	r_{(m+1)p}=\pr r_{mp} - \eta_m
\end{equation}
where $r_n$ is the $n$-th remainder from the greedy algorithm computing $\DB(x)$.

\begin{proposition}
\label{Pro : Equivalences1}
Let $x\in [0,1)$. The following assertions are equivalent.
\begin{enumerate}
\item The sequence $(\varepsilon_n)_{n\in\N}$ is ultimately periodic, i.e., $x\in \Per(\B)$.
\item The sequence $(\eta_m)_{m\in\N}$ is ultimately periodic.
\item There exist $m,m'\in\N$ such that $r_{mp}=r_{m'p}$, i.e., the sequence $(r_{mp})_{m\in\N}$ is not injective.
\end{enumerate}
\end{proposition}

\begin{proof}
If the sequence $(\varepsilon_n)_{n\in\N}$ is ultimately periodic with period $t$, then so is the sequence $(\eta_m)_{m\in\N}$. Hence the implication $(1) \implies (2)$ is verified. The implication $(3) \implies (1)$ follows from the greedy algorithm. Indeed, if there exist $m,m'\in\N$ such that $r_{mp}=r_{m'p}$, then $\varepsilon_{mp+n}=\varepsilon_{m'p+n}$ for all $n\in\N$. It remains to prove $(2)\implies (3)$. By~\eqref{Eq : x sum eta} and~\eqref{Eq : rn+1 rn}, for all $m\in \N$, we have 
\begin{align}
\label{Eq : Prop5}
	r_{mp}
	&=\pr^m\left(x-\sum_{\ell=0}^{m-1}\frac{\eta_\ell}{\pr^{\ell+1}}\right)\\
\nonumber
	&=\pr^m\left(\sum_{\ell=m}^\infty\frac{\eta_\ell}{\pr^{\ell+1}}\right).
\end{align}
Suppose that the sequence $(\eta_m)_{m\in \N}$ is ultimately periodic, i.e., that there exists $s\in \N$ and $t\in \N_{\ge1}$ such that $\eta_\ell=\eta_{\ell+t}$ for all $\ell\ge s$. Then 
\[
	r_{sp}
	=\pr^s\left(\sum_{\ell=s}^\infty\frac{\eta_{\ell+t}}{\pr^{\ell+1}}\right)
	=\pr^{s+t}\left(\sum_{\ell=s+t}^\infty\frac{\eta_\ell}{\pr^{\ell+1}}\right)
	=r_{(s+t)p}.
\]
\end{proof}

\begin{remark}
\label{rem:thue-morse}
We note that grouping terms $p$ by $p$ in an arbitrary $\B$-representation (not necessarily the greedy one), we do not get such equivalences. More precisely, it might be that, when dealing with other $\B$-representations than the greedy one, the considered representation is not ultimately periodic whereas the corresponding sequence of digits obtain by grouping $p$ by $p$ as in~\eqref{Eq : eta} is ultimately periodic. For example, for the alternate base $\B=(\varphi,\varphi,\varphi)$ where $\varphi$ is the Golden ratio $\frac{1+\sqrt{5}}{2}$, consider the infinite word $f(t)$ where $t$ is the famous Thue-Morse infinite word $01101001\cdots$ and $f\colon \{0,1\}^*\to \{0,1\}^*$ is the injective coding defined by $f(0)=100$ and $f(1)=011$. This word $f(t)$ is a $\B$-representation of $\frac{\varphi}{2}$. Since the Thue-Morse word is not ultimately periodic, the word $f(t)$ is not ultimately periodic either. Since $f_{\B}(0,1,1)=f_{\B}(1,0,0)$, by grouping the terms $3$ by $3$, we get the purely periodic infinite word $(\varphi^2)^\omega$ over the alphabet $\Dig(\B)=\{0,1,\varphi,\varphi+1,\varphi^2,\varphi^2+1,\varphi^2+\varphi,\varphi^2+\varphi+1\}=\{0,1,\varphi,\varphi^2,\varphi^2+1,\varphi^3,2\varphi^2\}$.
\end{remark}

\subsection{$\B$-expansions of rational numbers}

For our purposes, we need to prove the following theorem, which is a generalization of~\cite[Theorems 14 and 19]{CharlierCisterninoMasakovaPelantova2023}. Compared with the statements therein, we now authorize alternate representations of numbers of the form $\frac{1}{q}$ for $q$ a non-zero integer instead of alternate representations of $1$ only, and the condition that the $p$ sequences all start with a positive digit is now relaxed.

\begin{theorem}
\label{Thm : GeneralizationPrague}
If for every $i\in \Int$, there exists a non-zero integer $q_i$ and an ultimately periodic sequence $a_i=(a_{i,n})_{n\in\N}$ of integers such that 
\begin{equation}
\label{Eq : val}
	\sum_{n=0}^\infty\frac{a_{i,n}}{\prod_{k=0}^n\beta_{i+k}}
	=\frac{1}{q_i},
\end{equation}
then $\pr$ is an algebraic integer. If moreover these $p$ sequences have non-negative elements and for all $i\in \Int$, there exists $m_i\in \N$ such that $a_{i,m_ip}\ge 1$, then $\beta_0,\ldots,\beta_{p-1} \in \Q(\pr)$.
\end{theorem}

This result can be proven by following the same lines as in the proof of~\cite[Theorem 14]{CharlierCisterninoMasakovaPelantova2023}. We only give a sketch of the proof underlying the main differences. 

\begin{proof}[Sketch of the proof of Theorem~\ref{Thm : GeneralizationPrague}]

We start with the first part of the statement. For every $i\in \Int$, let $q_i$ be a non-zero integer and let $a_i=(a_{i,n})_{n\in\N}$ be an ultimately periodic sequence of integers such that~\eqref{Eq : val} holds. Without loss of generality, we suppose that the $p$ sequences $a_i$ all have preperiod $mp$ and period $kp$ with $m\in \N$ and $k\in \N_{\ge 1}$. Then, for all $i,j\in\Int$, we define $g_{i,j}$ as the following polynomial in the indeterminate $X$:
\[
	g_{i,j}
	=q_i(X^k-1)\left(\sum_{n=0}^{m-1} a_{i,j+np}X^{m-1-n}\right) 
	+ q_i\sum_{n=0}^{k-1} a_{i,j+(m+n)p}X^{k-1-n}.
\] 
(We note that whenever $q_i=1$, the polynomials $g_{i,0},\ldots,g_{i,p-1}$ coincide with those from the proof of~\cite[Theorem 14]{CharlierCisterninoMasakovaPelantova2023}.) As in~\cite[Lemma 17]{CharlierCisterninoMasakovaPelantova2023}, we have
\begin{equation}
\label{Eq : Lemme17}
  	\pr^m(\pr^k-1) 
  	= \sum_{j=0}^{p-1}g_{i,j}(\pr)\beta_{i+j+1}\cdots\beta_{i+p-1}.
\end{equation}
Then we consider the associated polynomial system
\[
	\begin{cases}
	X^m(X^k-1)
	= \displaystyle{\sum_{j=0}^{p-1}}g_{i,j}X_{i+j+1}\cdots X_{i+p-1}, & \text{for } i\in\Int \\
	X=X_0\ldots X_{p-1}
	\end{cases}
\] 
in the $p+1$ indeterminates $X_0,\ldots,X_{p-1},X$, with the convention that $X_{n+p}=X_n$ for all integers $n$. As in the proof of~\cite[Theorem 14]{CharlierCisterninoMasakovaPelantova2023}, by rearranging the equations and by substituting $X_0\cdots X_{p-1}$ by $X$, the first $p$ equations of the system can be written in the matrix form
\begin{equation}
\label{Eq : SystemForAlgebraic}
(M-X^m(X^k-1)I_p) \begin{pmatrix}
	X_1X_2\cdots X_{p-1}\\ 
	X_2\cdots X_{p-1}\\ 
	\vdots\\
	X_{p-1}\\
	1
	\end{pmatrix}
	=
	\begin{pmatrix}
	0\\ 
	\vdots\\
	0
	\end{pmatrix}
\end{equation}
where 
\[
	M=
	\begin{pmatrix}
	g_{1,p-1} & Xg_{1,0}	  & \cdots &	Xg_{1,p-3} & Xg_{1,p-2} \\
	g_{2,p-2} & g_{2,p-1} & \cdots & Xg_{2,p-4} & Xg_{2,p-3} \\
	\vdots   & \vdots 	& \ddots & \vdots & \vdots \\
	g_{p-1,1} 	& g_{p-1,2}   	& \cdots & g_{p-1,p-1}   & Xg_{p-1,0} \\
	g_{0,0}   	& g_{0,1}     	& \cdots & g_{0,p-2}     & g_{0,p-1}
	\end{pmatrix}
\]
and $I_p$ is the identity matrix of size $p$. The fact that the product $\pr$ is an algebraic integer is obtained by using the same argument as in the proof of~\cite[Theorem 14]{CharlierCisterninoMasakovaPelantova2023}, namely, that $\beta$ is a root of the polynomial in $X$ obtained by computing the determinant of the matrix $M-X^m(X^k-1)I_p$, which has leading coefficient $(-1)^p$. 

Now we turn to the second part of the statement. Thus, from now on, we suppose that the sequences $a_i$ have non-negative digits and for all $i\in \Int$, let $m_i\in \N$ be such that $a_{i,m_ip}\ge 1$. By~\cite[Lemma 16]{CharlierCisterninoMasakovaPelantova2023}, we know that 
\[
	g_{i,j}(\pr)
 	=q_i\pr^m(\pr^k-1)
 	\sum_{n=0}^{\infty}\frac{a_{i,j+np}}{\pr^{n+1}}
\]
for all $i,j\in\Int$. We get that the matrix $M(\pr)$ (that is, the matrix $M$ where the indeterminate $X$ is substituted by $\pr$ in every entry) has non-negative entries. Moreover, for all $i\in\Int$, we have $g_{i,0}(\pr)>0$ since $a_{i,m_ip}\ge 1$. Thus the entries $\pr g_{1,0}(\pr)$, $\pr g_{2,0}(\pr),\ldots,\pr g_{p-1,0}(\pr)$, $g_{0,0}(\pr)$ of $M(\pr)$ in respective positions $(0,1)$, $(1,2),\ldots,(p-2,p-1)$, $(p-1,0)$ are all positive. Therefore, the matrix $M(\pr)$ is irreducible. Then, in the exact same way as in the last part of the proof of~\cite[Theorem 14]{CharlierCisterninoMasakovaPelantova2023}, we obtain that $\beta_0,\ldots,\beta_{p-1}\in \Q(\pr)$ by using the Perron-Frobenius theorem.  
\end{proof}

Note that the previous theorem is immediate while considering Renyi's expansions. In the alternate base framework, Theorem~\ref{Thm : GeneralizationPrague} (and~\cite[Theorems 14 and 19]{CharlierCisterninoMasakovaPelantova2023}) needs much more work.

\subsection{Field isomorphism on infinite sums}

Under the assumption of periodicity of the coefficients of a power series, and provided that $\beta$ and $\gamma$ are Galois conjugates of moduli greater than one, the $\Q$-isomorphism from $\Q(\beta)$ to $\Q(\gamma)$ mapping $\beta$ to $\gamma$ behaves well with respect to this series.

\begin{lemma}
\label{Lem:psi-periodique}
Let $\beta$ be an algebraic number, let $\gamma$ be a Galois conjugate of $\beta$ with $|\beta|,|\gamma|>1$, and let $ \psi$ be the $\Q$-isomorphism from $\Q(\beta)$ to $\Q(\gamma)$ defined by $\psi(\beta)=\gamma$. For all ultimately periodic sequences $(z_m)_{m\in\N}$ over $\Q(\beta)$, we have 
\[
	\psi\left(\sum_{m=0}^\infty\frac{z_m}{\pr^{m+1}}\right)
	=\sum_{m=0}^\infty\frac{\psi(z_m)}{\gamma^{m+1}}.
\]
\end{lemma}

\begin{proof}
Let $(z_m)_{m\in\N}$ be an ultimately periodic sequence over $\Q(\beta)$, say of preperiod $s$ and period $t$. Then 
\begin{align*}
	\psi\left(\sum_{m=0}^\infty\frac{z_m}{\pr^{m+1}}\right)
	&= \psi\left(\sum_{m=0}^{s-1}\frac{z_m}{\pr^{m+1}}		
	+\left(\sum_{m=s}^{s+t-1} \frac{z_m}{\pr^{m+1}}\right)
	\frac{\pr^t}{\pr^t-1}\right)\\
	&= \sum_{m=0}^{s-1}\frac{\psi(z_m)}{\gamma^{m+1}}		
	+\left(\sum_{m=s}^{s+t-1} \frac{\psi(z_m)}{\gamma^{m+1}}\right)
	\frac{\gamma^t}{\gamma^t-1}\\
	&=\sum_{m=0}^\infty\frac{\psi(z_m)}{\gamma^{m+1}}.
\end{align*}
\end{proof}

We note that this lemma will be used for algebraic integers $\beta>1$ only, but we have chosen to state it in its full generality.

\section{Proof of Theorem~\ref{Thm : OurSchmidt}}
\label{Sec:OurSchmidt}

We are now ready to prove our main result.

\begin{proof}[Proof of Theorem~\ref{Thm : OurSchmidt}]
We start with the first item. Suppose that $\Q\cap [0,1)\subseteq \bigcap_{i=0}^{p-1}\Per(\B^{(i)})$. For each $i\in\Int$, we can choose a sufficiently large integer $m_i$ such that $\max\{\frac{\beta_i\beta^{m_i}}{2},\beta^{m_i}\}+1<\beta_i\beta^{m_i}$. This yields that the intervals $(\beta^{m_i},\beta_i\beta^{m_i}]$ and $(\frac{\beta_i\beta^{m_i}}{2},\beta_i\beta^{m_i}]$ have lengths greater than one. Therefore, there exists an integer $q_i$ that belongs to both of them. Thus, for each $i\in\Int$, we have constructed a rational number of the form $\frac{1}{q_i}$ having a $\B^{(i)}$-expansion beginning with the prefix $0^{m_ip}1$. Moreover, by hypothesis, these $\B^{(i)}$-expansions are all ultimately periodic sequences of integers.
Therefore, we may apply Theorem~\ref{Thm : GeneralizationPrague}. We get that $\pr$ is an algebraic integer and that $\beta_0,\ldots,\beta_{p-1}\in \Q(\pr)$. Now, let $\gamma$ be a Galois conjugate of $\pr$ such that $|\gamma|>1$ and let $\psi$ be the $\Q$-isomorphism from $\Q(\pr)$ to $\Q(\gamma)$ defined by $\psi(\pr)=\gamma$. 

First, we claim that for any $x\in \Q\cap [0,1)$, we have 
\begin{equation}
\label{Eq : claim1}
	x=\sum_{m=0}^\infty\frac{\psi(\eta_m)}{\gamma^{m+1}}.
\end{equation} 
Consider $x\in \Q\cap [0,1)$. By hypothesis, we have $x\in \Per(\B)$. Hence, by Proposition~\ref{Pro : Equivalences1}, the sequence $(\eta_m)_{m\in\N}$ is ultimately periodic. Since $x=\psi(x)$ and since $\beta>1$ and $|\gamma|>1$, the claim follows by using Lemma~\ref{Lem:psi-periodique}.
 
Now, for all integers $M$ large enough so that $\frac{1}{\pr}+\frac{1}{\pr^M}<\frac{1}{\beta_0\cdots\beta_{p-2}}$ and for all real numbers $x$ in the interval $[\frac{1}{\pr},\frac{1}{\pr}+\frac{1}{\pr^M})$, we have $\eta_0=f_{\B}(0,\ldots,0,1)=1$ and $\eta_1=\cdots=\eta_{M-1}=f_{\B}(0,\ldots,0)=0$. Any interval of the form $[\frac{1}{\pr},\frac{1}{\pr}+\frac{1}{\pr^M})$ contains a rational number $x$, and for each such $x$, we get from~\eqref{Eq : x sum eta} and~\eqref{Eq : claim1} together with the previous observation that
\[	
	\frac{1}{\pr}+ \sum_{m=M}^\infty\frac{\eta_m}{\pr^{m+1}}
	=\frac{1}{\gamma}+\sum_{m=M}^\infty\frac{\psi(\eta_m)}{\gamma^{m+1}}.
\]
Setting $C=\max\{\eta : \eta\in\Dig(\B)\}$ and $D=\max\{|\psi(\eta)| : \eta\in\Dig(\B)\}$, we get that
\[
	\left|\frac{1}{\gamma}-\frac{1}{\pr}\right|
	\le \sum_{m=M}^\infty\frac{\eta_m}{\pr^{m+1}}
	+ \left|\sum_{m=M}^\infty\frac{\psi(\eta_m)}{\gamma^{m+1}}\right|
	\le \frac{C}{\pr^M(\pr-1)} + \frac{D}{|\gamma|^M(|\gamma|-1)}
\]
for all large enough $M$ and all $x\in\Q\cap [\frac{1}{\pr},\frac{1}{\pr}+\frac{1}{\pr^M})$. Since the last upper bound is independent of the chosen $x$, by letting $M$ tend to infinity we obtain that $\gamma=\pr$. This proves that $\pr$ is either a Pisot number or a Salem number.
\medskip

We now turn to the second item of the statement. Suppose that $\pr$ is a Pisot number and that the bases $\beta_0,\ldots,\beta_{p-1}$ belong to $\Q(\pr)$. 

First, we prove that $\Per(\B)\subseteq \Q(\pr)\cap [0,1)$. Let $x\in \Per(\B)$. By Proposition~\ref{Pro : Equivalences1}, the sequence $(\eta_m)_{m\in \N}$ is ultimately periodic, say with preperiod $s$ and period $t$. Then
\[
	x
	=\sum_{m=0}^\infty\frac{\eta_m}{\pr^{m+1}}
	=\sum_{m=0}^{s-1}\frac{\eta_m}{\pr^{m+1}} 
	+ \left(\sum_{m=s}^{s+t-1}\frac{\eta_m}{\pr^{m+1}}\right) 
	\frac{\pr^t}{\pr^t-1}.
\]
Since $\Q(\pr)$ is a field and the digits $\eta_m$ belong to $\Q(\pr)$, we get that $x\in \Q(\pr)\cap [0,1)$. 

Second, we turn to the converse inclusion. Let $x\in \Q(\pr)\cap [0,1)$. Proceed by contradiction and suppose that $x\notin\Per(\B)$. Then the sequence $(r_{mp})_{m\in\N}$ is injective by Proposition~\ref{Pro : Equivalences1}. By~\eqref{Eq : Prop5}, we obtain that $r_{mp}$ belongs to the set 
\[
	X^{\Delta}(\pr)
	=\left\{\sum_{\ell=0}^{m-1}a_\ell\pr^{m-1-\ell}
	: m\in\N,\ a_0,\ldots,a_{m-1}\in\Delta\right\}
\] 
for all $m\in\N$, where $\Delta=-\Dig(\B)\cup\Dig(\B)\cup\{x\}$. The set $X^{\Delta}(\pr)$ is called the $\beta$-\emph{spectrum} over the real digit set $\Delta$. Since the remainders $r_{mp}$ belong to $[0,1)$, the spectrum $X^{\Delta}(\pr)$ has an accumulation point in $\R$. By~\cite[Proposition 24]{CharlierCisterninoMasakovaPelantova2023}, which is indeed valid for all finite digit sets included in $\Q(\pr)$, we get that either $\pr$ is not a Pisot number or there exists $i\in\Int$ such that $\beta_i\notin\Q(\pr)$. But this is in contradiction with our hypotheses.
\end{proof}

We now show how Corollary~\ref{Cor : PisotimpliqueParry} may be deduced from Theorem~\ref{Thm : OurSchmidt}.

\begin{proof}[Proof of Corollary~\ref{Cor : PisotimpliqueParry}]
Suppose that $\beta$ is a Pisot number and that $\beta_0,\ldots,\beta_{p-1}\in\Q(\beta)$.  Although Theorem~\ref{Thm : OurSchmidt} deals with real numbers in $[0,1)$, we may easily use it in order to show the ultimate periodicity of $\qDBi{i}(1)$ for all $i\in\Int$. Equivalently, we show that $\DBi{i}(1)$ is ultimately periodic for all $i\in\Int$. Since $\DBi{i}(1)=\floor{\beta_i}\DBi{i+1}(\beta_i-\floor{\beta_i})$, this property is equivalent to the fact that $\beta_i-\floor{\beta_i}\in\Per(\B^{(i+1)})$ for all $i\in\Int$. But from the hypotheses, we know that $\beta_i-\floor{\beta_i}\in \Q(\pr)\cap[0,1)$ for all $i\in\Int$. Hence the result follows from Theorem~\ref{Thm : OurSchmidt}. \end{proof}

\begin{remark}
Let us illustrate that the item~\eqref{OurSchmidt:2} of Theorem~\ref{Thm : OurSchmidt} cannot be extended to non-greedy representations, i.e., if $\pr$ is a Pisot number and $\beta_0,\ldots,\beta_{p-1} \in \Q(\pr)$ then the set of real numbers in $[0,1)$ having ultimately periodic $\B$-representations with respect to an algorithm different from the greedy one need not be equal to $\Q(\pr)\cap [0,1)$. We consider again the example introduced in Remark~\ref{rem:thue-morse}.  We know that $\frac{\varphi}{2}=\frac{\varphi^3-1}{4}\in\Q(\varphi^3)$. However, the infinite word $f(t)$ is a non ultimately periodic $\B$-representation  of $\frac{\varphi}{2}$ although $\varphi^3$ is a Pisot number. 
\end{remark}

\section{Proof of Theorem~\ref{Thm:stronger}}
\label{Sec:stronger}

Finally, we are able to obtain a much stronger statement than the first item of Theorem~\ref{Thm : OurSchmidt}.

\begin{proof}[Proof of Theorem~\ref{Thm:stronger}]
Suppose that $\pr$ is an algebraic integer that is neither a Pisot number nor a Salem number. Thus, there exists a Galois conjugate $\gamma$ of $\pr$ such that $\gamma\ne\pr$ and $|\gamma|>1$. For $x\in[0,1)$, define 
\[
	E(x)=\left\{n\in\N : 
	\sum_{m=0}^{n-1}\frac{\eta_m(x)}{\pr^{m+1}}
	\ne \sum_{m=0}^{n-1}\frac{\psi(\eta_m(x))}{\gamma^{m+1}}
	\right\}
\] 
where the digits $\eta_m(x)$ are defined as in~\eqref{Eq : x sum eta} and $\psi$ is the $\Q$-isomorphism from $\Q(\pr)$ to $\Q(\gamma)$ such that $\psi(\pr)=\gamma$. 

Let us first show that the set
\[
	F=\{x\in[0,1) : E(x)\text{ is infinite}\}.
\] 
is dense in $[0,1)$. Let $x$ be in $[0,1)$ and let $\varepsilon$ be a positive real number. We construct an element $y$ of $F$ such that $|x-y|<\varepsilon$. Let $c,d\in\N$ be such that $c<d$, $\pr^d\neq \gamma^d$, $\frac{1}{\pr^d}<\frac{\varepsilon}{2}$ and $\frac{C}{\pr^c(\pr-1)}<\frac{\varepsilon}{2}$ where $C=\max\{\eta : \eta\in\Dig(\B)\}$. As a first case, suppose that $c\in E(x)$. In this case, we let
\[
	y=\sum_{m=0}^{c-1} \frac{\eta_m(x)}{ \pr^{m+1}}.
\] 
Then $\eta_m(y)=\eta_m(x)$ for all $m\in[\![0,c-1]\!]$ and $\eta_m(y)=0$ for all $m\ge c$. Therefore, the set $E(y)$ contains all integers $n\ge c$, and hence $y\in F$. Now, suppose that $c\notin E(x)$. Then we let
\[
	y=\sum_{m=0}^{c-1} \frac{\eta_m(x)}{ \pr^{m+1}}+\frac{1}{\pr^d}.
\] 
In this case, for $m\in[\![0,c-1]\!]$, we have $\eta_m(y)=\eta_m(x)$ and for $m\ge c$, we have $\eta_m(y)=0$ if $m\ne d-1$ and $\eta_{d-1}(y)=1$. We obtain that $n\in E(y)$ for all $n\ge d$, hence $y\in F$ in this case as well. In both cases, we have found the desired element $y$ of $F$ since 
\[
	|x-y|
	\le \sum_{m=c}^{\infty} \frac{\eta_m(x)}{ \pr^{m+1}}+\frac{1}{\pr^d}
	\le\frac{C}{\pr^c(\pr-1)}+\frac{1}{\pr^d}
	<\frac{\varepsilon}{2}+\frac{\varepsilon}{2}
	=\varepsilon.
\]
 
Now, since for all $x\in F$ and $\varepsilon>0$, there exists $n\in E(x)$ such that 
\[
	0\le x-\sum_{m=0}^{n-1}\frac{\eta_m(x)}{\pr^{m+1}}<\varepsilon,
\]
we get that the set
\[
	G=\left\{\sum_{m=0}^{n-1}\frac{\eta_m(x)}{\pr^{m+1}} : 
	x\in [0,1),\ n\in E(x)\right\}
\]
is dense in $[0,1)$ as well. In order to see that $\Per(\B)\cap\Q$ is nowhere dense in $[0,1)$, it then suffices to show that for all $g\in G$, there exists $N\in\N$ such that 
\[
	\left[g,g+\frac{1}{\pr^N}\right)\cap \Per(\B)\cap\Q\cap[0,1)=\emptyset.
\]

Proceed by contradiction and suppose that there exist $x\in [0,1)$ and $n\in E(x)$ such that for all $N\in\N$, there exists $y\in\Per(\B)\cap\Q\cap[0,1)$ such that 
\begin{equation}
\label{Eq : y2}
	0\le y-\sum_{m=0}^{n-1}\frac{\eta_m(x)}{\pr^{m+1}}<\frac{1}{\pr^N}.
\end{equation}
We consider such $x$ and $n$. Let now  $N\ge n$ and $y\in\Per(\B)\cap\Q\cap[0,1)$ such that~\eqref{Eq : y2} holds. Then $\eta_m(y)=\eta_m(x)$ for all $m\in[\![0,n-1]\!]$ and $\eta_m(y)=0$ for all $m\in[\![n,N-1]\!]$. As in the proof of the item~\eqref{OurSchmidt:1} of Theorem~\ref{Thm : OurSchmidt}, we obtain that
\[
	y=\sum_{m=0}^\infty \frac{\psi(\eta_m(y))}{\gamma^{m+1}}.
\]
We then get
\[
	\left|\sum_{m=0}^{n-1} \frac{\eta_m(x)}{\pr^{m+1}}
	-\sum_{m=0}^{n-1} \frac{\psi(\eta_m(x))}{\gamma^{m+1}}\right|
	=\left|\sum_{m=N}^\infty \frac{\eta_m(y)}{\pr^{m+1}}
	-\sum_{m=N}^\infty \frac{\psi(\eta_m(y))}{\gamma^{m+1}}\right|
	\le\frac{C}{\pr^N(\pr-1)}
	+\frac{D}{|\gamma|^N(|\gamma|-1)}
\]
where $C=\max\{\eta : \eta\in\Dig(\B)\}$ and 
$D=\max\{|\psi(\eta)| : \eta\in\Dig(\B)\}$. Since $N$ can be chosen arbitrarily large, we obtain that 
\[
	\sum_{m=0}^{n-1} \frac{\eta_m(x)}{\pr^{m+1}}
	=\sum_{m=0}^{n-1} \frac{\psi(\eta_m(x))}{\gamma^{m+1}},
\]
contradicting that $n\in E(x)$.
\end{proof}

\section{Reduction to the case of real base expansions}
\label{Sec:p=1}

In this section, we emphasize that our proof of Schmidt's results, i.e., when taking $p=1$ in Theorem~\ref{Thm : OurSchmidt}, is in fact much simpler that the original one, which relies on several technical lemmas; see~\cite[Theorems 2.4 and 3.1]{Schmidt1980}. In order to do this, we explicitly give the reduction of our proof to this case.

\begin{proof}[New proof of Schmidt's theorems]
We consider the case where $p=1$ only. So the notation  $\B$, $\beta_0$ and $\beta$ now coincide; so we simply  write $\beta$ for any of them. We start with the first item. Suppose that $\Q\cap [0,1)\subseteq \Per(\beta)$. In particular, there exists an ultimately periodic sequence of integers whose $\beta$-value is of the form $\frac{1}{q}$ with $q\in\N_{\ge 1}$. Therefore, $\beta$ satisfies an equality of the form
\[
	\frac{1}{q}=\sum_{n=0}^{s-1}\frac{a_n}{\beta^{n+1}}		
	+\left(\sum_{n=s}^{s+t-1} \frac{a_n}{\beta^{n+1}}\right)
	\frac{\beta^t}{\beta^t-1}
\] 
where $s\in\N$ and $t\in\N_{\ge 1}$. Multiplying the two sides of this equality by $q\beta^s(\beta^t-1)$, we see that $\beta$ is an algebraic integer. Now, let $\gamma$ be a Galois conjugate of $\beta$ such that $|\gamma|>1$ and let $\psi$ be the $\Q$-isomorphism from $\Q(\beta)$ to $\Q(\gamma)$ defined by $\psi(\beta)=\gamma$. 

First, we claim that for any $x\in \Q\cap [0,1)$, we have 
\begin{equation}
\label{Eq : claim2}
	x=\sum_{n=0}^\infty\frac{\varepsilon_n}{\gamma^{n+1}}.
\end{equation} 
Consider $x\in \Q\cap [0,1)$. By hypothesis, we have $x\in \Per(\B)$. Hence the sequence $(\varepsilon_n)_{n\in\N}$ is ultimately periodic. 
Since $x=\psi(x)$ and since $\beta>1$ and $|\gamma|>1$, the claim follows by using Lemma~\ref{Lem:psi-periodique}.

Now, observe that if a real number $x$ belongs to an interval of the form $[\frac{1}{\beta},\frac{1}{\beta}+\frac{1}{\beta^N})$ with $N\in\N_{\ge 1}$, then $\varepsilon_0=1$ and $\varepsilon_1=\cdots=\varepsilon_{N-1}=0$.  For all $N\in\N_{\ge 1}$, the interval $[\frac{1}{\beta},\frac{1}{\beta}+\frac{1}{\beta^N})$ contains a rational number $x$, and for each such $x$, we get from~\eqref{Eq : claim2} and the previous observation that
\[	
	\frac{1}{\beta}+ \sum_{n=N}^\infty\frac{\varepsilon_n}{\beta^{n+1}}
	=\frac{1}{\gamma}+\sum_{n=N}^\infty\frac{\varepsilon_n}{\gamma^{n+1}}.
\]
We get that
\[
 	\left|\frac{1}{\gamma}-\frac{1}{\beta}\right|
	\le \sum_{n=N}^\infty\frac{\varepsilon_n}{\beta^{n+1}}
	+ \left|\sum_{n=N}^\infty\frac{\varepsilon_n}{\gamma^{n+1}}\right|
	\le \frac{\ceil{\beta}-1}{\beta^N(\beta-1)} + \frac{\ceil{\beta}-1}{|\gamma|^N(|\gamma|-1)}
\]
for all $N\in\N_{\ge 1}$ and $x\in\Q\cap [\frac{1}{\beta},\frac{1}{\beta}+\frac{1}{\beta^N})$. Since the right-most upper bound does not depend on the chosen $x$, by letting $N$ tend to infinity we obtain that $\gamma=\beta$. This proves that $\beta$ is either a Pisot number or a Salem number.
\medskip

We now turn to the second item of the statement. Suppose that $\beta$ is a Pisot number.

First, we prove that $\Per(\beta)\subseteq \Q(\pr)\cap [0,1)$. Let $x\in \Per(\beta)$. Then $(\varepsilon_n)_{n\in \N}$ is ultimately periodic, say with preperiod $s$ and period $t$. Then
\[
	x
	=\sum_{n=0}^\infty\frac{\varepsilon_n}{\beta^{n+1}}
	=\sum_{n=0}^{s-1}\frac{\varepsilon_n}{\beta^{n+1}} 
	+ \left(\sum_{n=s}^{s+t-1}\frac{\varepsilon_n}{\beta^{n+1}}\right) 
	\frac{\beta^t}{\beta^t-1}.
\]
Since $\Q(\pr)$ is a field and the digits $\varepsilon_n$ are integers, we get that $x\in \Q(\pr)\cap [0,1)$. 

Second, we turn to the converse inclusion. Let $x\in \Q(\pr)\cap [0,1)$. Proceed by contradiction and suppose that $x\notin\Per(\B)$. Then the sequence of remainders $(r_n)_{n\in\N}$ is injective. By~\eqref{Eq : Prop5}, we obtain that $r_n$ belongs to the spectrum 
\[
	X^{\Delta}(\beta)
	=\left\{\sum_{\ell=0}^{m-1}a_\ell\beta^{m-1-\ell}
	: m\in\N,\ a_0,\ldots,a_{m-1}\in\Delta\right\}
\] 
for all $n\in\N$, where $\Delta=[\![-\ceil{\beta}+1,\ceil{\beta}-1]\!]\cup\{x\}$. Since the remainders $r_n$ belong to $[0,1)$, the spectrum $X^{\Delta}(\pr)$ has an accumulation point in $\R$. By~\cite[Proposition 24]{CharlierCisterninoMasakovaPelantova2023}, 
which is indeed valid for all finite digit sets included in $\Q(\pr)$, 
we get that $\beta$ is not a Pisot number. But this is in contradiction with our hypothesis.
\end{proof}

\section{Open questions}
\label{Sec:future}

Schmidt conjectured that for real bases $\beta$ that are Salem numbers, we still have that $\Per(\pr)=\Q(\pr)\cap[0,1)$~\cite{Schmidt1980}. Even though some work has been done, this problem is still open today, as well as the partial problem to know whether all Salem numbers are Parry numbers; see for instance \cite{Boyd1989,Boyd1996,Hichri2014,Vavra2021}. Analogous questions can be asked in the alternate base framework. Namely, is it true that if the product of the bases $\pr=\prod_{i=0}^{p-1}\beta_i$ is a Salem number and if $\beta_0,\ldots,\beta_{p-1} \in \Q(\pr)$, then $ \Per(\B)=\Q(\pr)\cap [0,1)$? And more specifically, is it true that under the same conditions, the alternate base $\B$ is Parry?

In the first item of Theorem~\ref{Thm : OurSchmidt}, our assumption concerns the $p$ shifts $\B^{(i)}$ of the alternate base $\B$. However, we need information about the shifts $\B^{(i)}$ for $i\ne 0$ only when using Theorem~\ref{Thm : GeneralizationPrague}. We leave as an open question to know whether we need information about all the shifted bases or not. If we were able to improve Theorem~\ref{Thm : GeneralizationPrague} by proving that if $\Q\cap [0,1)\subseteq \Per(\B)$ then $\beta_0,\ldots,\beta_{p-1} \in \Q(\pr)$ and $\pr$ is an algebraic integer, then we would obtain that if $\Q\cap [0,1)\subseteq \Per(\B)$ then $\beta_0,\ldots,\beta_{p-1} \in \Q(\pr)$ and $\pr$ is either a Pisot number or a Salem number. Otherwise stated, an improvement of Theorem~\ref{Thm : GeneralizationPrague} would give rise to an   improvement of Theorem~\ref{Thm : OurSchmidt}.

\section{Acknowledgment}
We thank the referee for helpful suggestions.
Émilie Charlier is supported by the FNRS grant J.0034.22.
Célia Cisternino is supported by the FNRS Research Fellow grant 1.A.564.19F.
Savinien Kreczman is supported by the FNRS Research Fellow grant 1.A.789.23F.

\bibliographystyle{abbrv}
\bibliography{Bibliography}

\end{document}